\documentclass[reqno]{amsart}
\usepackage[scale=0.75, centering, headheight=14pt]{geometry}
\usepackage[latin1]{inputenc}
\usepackage[T1]{fontenc}
\usepackage{lmodern}
\usepackage[english]{babel}
\usepackage{esint}

\usepackage{tikz}
\usepackage{bbm}
\usepackage{amsmath,amssymb,amsfonts,amsthm}
\usepackage{mathtools,accents}
\usepackage{mathrsfs}
\usepackage{xfrac}
\usepackage{array} 
\usepackage{aliascnt}
\usepackage{booktabs} 
\usepackage{array} 

\usepackage{verbatim} 
\usepackage{subfig} 

\usepackage{mathrsfs, dsfont}
\usepackage{amssymb}
\usepackage{amsthm}
\usepackage{amsmath,amsfonts,amssymb,esint}
\usepackage{graphics,color}
\usepackage{enumerate}
\usepackage{mathtools,centernot}

\usepackage{microtype}
\usepackage{paralist} 
\usepackage{cases}
\usepackage[initials]{amsrefs}
\allowdisplaybreaks

\usepackage{braket}
\usepackage{bm}

\usepackage[citecolor=blue,colorlinks]{hyperref}
\addto\extrasenglish{}

\usepackage{enumerate}
\usepackage{xcolor}
\usepackage{aliascnt}

\makeatletter
\def\newaliasedtheorem#1[#2]#3{
  \newaliascnt{#1@alt}{#2}
  \newtheorem{#1}[#1@alt]{#3}
  \expandafter\newcommand\csname #1@altname\endcsname{#3}
}
\makeatother

\theoremstyle{plain}
\newtheorem{theorem}{Theorem}[section]
\newaliasedtheorem{lemma}[theorem]{Lemma}
\newaliasedtheorem{prop}[theorem]{Proposition}
\newaliasedtheorem{claim}[theorem]{Claim}
\newaliasedtheorem{corollary}[theorem]{Corollary}
\theoremstyle{remark}

\theoremstyle{definition}
\newaliasedtheorem{definition}[theorem]{Definition}
\newaliasedtheorem{example}[theorem]{Example}

\theoremstyle{remark}
\newaliasedtheorem{remark}[theorem]{Remark}

\numberwithin{equation}{section}

\def\eps{\epsilon}

\def\R{\mathbb R}
\def\N{{\mathbb N}}
\def\Z{{\mathbb Z}}
\def\T{{\mathbb T}}
\def\Q{{\mathbb Q}}

\DeclareMathOperator{\diver}{div}
\DeclareMathOperator{\curl}{curl}

\newcommand*{\RR}{\ensuremath{\mathcal{R}}}

\newcommand{\PH}{\mathbb{P}_H}
\newcommand{\PP}{\mathbb{P}_{\neq 0 }}

\DeclareMathOperator{\Id}{Id}

\title{Typicality results for weak solutions of the incompressible Navier--Stokes equations}

\author[M. Colombo, L. De Rosa and  M. Sorella]{Maria Colombo, Luigi De Rosa \and Massimo Sorella}
\address{Maria Colombo
\hfill\break  \'Ecole Polytechnique F\'ed\'erale de Lausanne, Institute of Mathematics, Station 8, CH-1015 Lausanne, Switzerland.}
\email{maria.colombo@epfl.ch}
\address{Luigi De Rosa 
\hfill\break  \'Ecole Polytechnique F\'ed\'erale de Lausanne, Institute of Mathematics, Station 8, CH-1015 Lausanne, Switzerland.}
\email{luigi.derosa@epfl.ch}
\address{Massimo Sorella
\hfill\break \'Ecole Polytechnique F\'ed\'erale de Lausanne, Institute of Mathematics, Station 8, CH-1015 Lausanne, Switzerland.}
\email{massimo.sorella@epfl.ch}

\begin{document}
\maketitle
\begin{abstract}
In the class of $L^\infty((0,T);L^2(\T^3))$
 distributional solutions of the incompressible Navier-Stokes system, the ones which are smooth in some open interval of times are meagre in the sense of Baire category, and the Leray ones are a nowhere dense set.
\end{abstract}
\par
\medskip\noindent
{\sc MSC (2010): 35Q30 - 35D30 - 76B03 - 26A21.
\par
}
\section{Introduction}\label{sec:intro}

In the last 15 years, the fundamental results of De Lellis and Sz\'ekelyhidi \cite{DLSZ09,DLSZ13,DLSZ17} initiated a research line which allowed to build nonsmooth distributional solutions of various equations in fluid dynamics with increasingly many regularity properties. All these results share a common approach called convex integration, which in this context points roughly speaking to build solutions of a nonlinear PDE by an iterative procedure, where at each step the constructed functions solve the equation up to a smaller and smaller error, which is corrected each time by means of the nonlinearity of the PDE. This lead to important results such as the proof of the Onsager conjecture by Isett \cite{Is18,BDSV19} and the construction of nonsmooth distributional solutions to the Navier-Stokes equations by Buckmaster and Vicol \cite{BV19,BCV,CL20}.
Related recent results were obtained for the hypodissipative Navier-Stokes equations \cite{CDLDR18,DR19} , the surface-quasigeostrophic equation \cite{BVS,IM20,CKL20} and the transport equation \cite{MS18,MS19,MSa20,BCD20} 
(see also the references quoted therein).\\

A natural question is then ``how many'' such distributional solutions can be found, compared to the smooth ones. 
In this paper we investigate this question in terms of Baire category.  We focus on the Navier-Stokes system in the spatial periodic setting $\T^3=\R^3 / \Z^3$
\begin{equation}\label{NS}
\left\{\begin{array}{l}
\partial_t v+  \diver(v \otimes v)+\nabla p-  \Delta v =0\\ 
 \diver v= 0\,
\end{array}\right.\qquad \mbox{in }  \T^3\times [0,T] 
\end{equation}
where $v:  \T^3\times [0,T]   \rightarrow \R^3$ represents the velocity of an incompressible fluid, $p:\T^3\times [0,T]  \rightarrow \R$ is the hydrodynamic pressure, with the constraint $\int_{\T^3}p\, dx =0$.\\

We define the following complete metric space
$$
\mathcal{D}:= \left\{ v \in L^\infty((0,T);L^2(\T^3)) \,: \, v\, \textit{ is a distributional solution of \eqref{NS}}  \right\},
$$
endowed with the metric $d_\mathcal{D}(u,v):=\|u-v\|_{L^\infty_t(L^2_x)}$, and its subsets
\begin{align*}
\mathcal{L}&:= \left\{v\in \mathcal{D} \, : \, v \textit{ is a Leray--Hopf solution of \eqref{NS}} \right\}\\
\mathcal{S}&:= \left\{v\in \mathcal{D} \, : \, v\in C^\infty(\T^3 \times I)  \textit{ for some open interval } I\subset(0,T) \right\}.
\end{align*}
We refer to Section~\ref{sec:defn} for the definitions of distributional and Leray-Hopf solutions. 
Our main result is the following
\begin{theorem}\label{t_main}
The set $\mathcal{L}$ is nowhere dense in $\mathcal{D}$ while the set $\mathcal{S}$ is meagre in $\mathcal{D}$.
\end{theorem}

We recall that $\mathcal{L}$ is nowhere dense in $\mathcal{D}$ if and only if the closure of $\mathcal{L}$ has empty interior. In particular, $\mathcal{L}$ is meagre in $\mathcal{D}$.\\

A partial answer to the question of ``how many'' distributional solutions there are, compared to the smooth ones, was given before by the so called ``h-principle'', a term introduced by Gromov in the context of isometric embeddings. In the context of the Euler equations (see for instance \cite[Theorem 6]{DS15}), it states that arbitrarily close in the weak $L^2$ topology to a (suitably defined) strict subsolution 
 one can build an exact distributional solution. In a slightly different direction, it has been shown in \cite{DanSz} that a dense set of initial data admits infinitely many distributional solutions with the same kinetic energy, and in \cite{BCV} that distributional solutions are nonunique for any initial datum in $L^2$ for the Navier-Stokes system.
Previously, convex integration was also used in \cite{DRT20} to characterize typical energy profiles for the Euler equations in terms of H\"older spaces, which requires to introduce a suitable metric space to deal with the right energy regularity.

\subsection*{Aknowledgements}
The authors aknowledge the support of the SNF Grant $200021\_182565$.

\section{The iterative proposition and proof of the main theorem}

The proof of Theorem~\ref{t_main} is based on an iterative proposition, typical of convex integration schemes and analogous to \cite[Section 7]{BV19} and \cite[Section 2]{BCV}; in analogy with the latter, also here we use intermittent jets (see Section 3 below) as the fundamental building blocks.
At difference to the previously cited works, we need to keep track of the kinetic energy in some intervals of time along the iteration in such a way to be able to prescribe it in the limit, and we also need to make sure with a simple use of time cutoffs that the support of the perturbation is localized in a converging sequence of enlarging sets. On the contrary, we don't use the cutoffs to obtain a small set of singular times for our limit, as was done in \cite{BCV}.



In turn the proof of Theorem~\ref{t_main} follows from the iterative proposition in this way: to show that the subset $\mathcal{L}$ is nowhere dense in the metric space $\mathcal{D}$, we prove that for every $v\in \mathcal{L}$ there are arbitrarily close elements which belong to $\mathcal{D}\setminus \mathcal{L}$. In Step 1 of the proof we reduce to such statement, where we choose elements in $\mathcal{D} \setminus \mathcal{L}$ by imposing locally increasing kinetic energy.\\

%



The method presented here to prove Theorem~\ref{t_main} is quite general in contexts where the convex integration scheme works and  should apply also to other contexts.

\subsection{Basic notations and definitions}\label{sec:defn}

We recall that a distributional solution of the system \eqref{NS} is a vector field $v\in L^2(\T^3\times (0,T);\R^3)$ such that 
$$
\int_0^T\int_{\T^3}\left( v\cdot \partial_t \varphi+ v\otimes v : \nabla \varphi+ v \cdot \Delta \varphi \right)\,dx dt=0,
$$
for all $\varphi\in C^\infty_c(\T^3\times (0,T);\R^3)$ such that $\diver \varphi=0$. 
The pressure does not appear in the distributional formulation because it can be recovered as the unique $0$-average solution of 
\begin{equation}
\label{eqn:p}
-\Delta p=\diver \diver (v\otimes v).
\end{equation}

A Leray Hopf solution of the system \eqref{NS} is a vector field $v \in L^2 ((0,T);H^1(\T^3)) \cap L^\infty((0,T);L^2(\T^3))$ and for a.e. $s \geq 0$ and for all $t \in [s, T]$ the following inequality holds
\begin{align} \label{leray_inequality}
 \int_{\T^3} \dfrac{|v(x,t)|^2}{2} dx  + \int_s^{t} \int_{\T^3} |\nabla v(x,\tau)|^2 dx d\tau \leq \int_{\T^3} \dfrac{|v(x,s)|^2}{2} dx.
\end{align} 
It is a classical result by Leray that Leray-Hopf solutions are smooth outside a closed set of times of Hausdorff dimension $1/2$, see for instance \cite{LP}.

\subsection{The Navier--Stokes--Reynolds system}
In this section, for every integer $q \geq 0$ we will highlight the construction of a solution $(v_q, p_q, \mathring{R}_q)$ to the Navier-Stokes-Reynolds system
\begin{align} \label{reynolds}
\left\{\begin{array}{l}
\partial_t v_q+  \diver(v_q \otimes v_q)+\nabla p_q-  \Delta v_q =\diver \mathring R_q\\ 
 \diver v_q= 0\,
\end{array}\right.
\end{align}
where the Reynolds stress $\mathring{R}_q$ is assumed to be a trace-free symmetric matrix valued function. Indeed for any matrix $A$ we will use the notation $\mathring{A}$ to remark the traceless property.

\subsection{Parameters}

Define the frequency parameter $\lambda_q \rightarrow + \infty$ and the amplitudes parameter $\delta_q \rightarrow 0^+$ by

\begin{align*}
\lambda_q = 2 \pi a^{(b^q)},
\\
\delta_q = \lambda_q^{-2 \beta}.
\end{align*}

The sufficiently large (universal) parameter $b$ is free, and so is the sufficiently small parameter $\beta= \beta(b)$. The parameter $a$ is chosen to be a sufficiently large multiple of the geometric constant $n_\ast$.
Moreover, we fix another parameter useful to prescribe a precise kinetic energy

\begin{align}
\epsilon_1 := \left ( \frac{\epsilon}{\sup_{\xi \in \Lambda } \| \gamma_\xi \|_{C^0} |\Lambda| C_0 4 (2 \pi)^3} \right )^2,
\end{align}
where $\sup_{\xi \in \Lambda } \| \gamma_\xi \|_{C^0}, |\Lambda|, C_0$ are all universal constants independent on $q$, more precisely:
 $\gamma_\xi$ are functions defined in Lemma \ref{l_geometric},
 $\Lambda$ is the finite set defined in Lemma \ref{l_geometric},
 $C_0$ is the constant given by Lemma \ref{l_decorellation},
 $\epsilon$ is a free constant that will be used in the proof of Theorem \ref{t_main}.

Moreover, we will use the intermittent jets (defined in Section \ref{int_jets}) to define the new velocity increment at step $q+1$. 
\subsection{Inductive estimates and iterative proposition}

We define new ``slow'' parameters, for all $q \geq 0$
 
\begin{align}
& s_q:= \left( \frac{s}{2} \right)^{q+1}, 
\\
& S_q := \sum_{i = 0}^{q}  s_i , 
\end{align}
for some fixed parameter $s>0$. By choosing $a_0(s)$ sufficiently large,
we will guarantee that

\begin{align*}
s_{q+1}^{-1} \ll \lambda_q,
\end{align*}
indeed $s_q^{-1}$ is a slow parameter compared to $\lambda_q$.
Moreover we define the local time interval, for some small number $s>0$, for all $q \geq 0$
\begin{align} \label{locale_intervallo_i_q}
I_q := (t_0 - S_q, t_0 + S_q),
\end{align}
for some $t_0 \in (0,1)$ and $s=s (t_0)>0$ sufficiently small such that 
$$
B_{2s}(t_0):= (t_0 - 2s, t_0 +2s) \subset [0,1].
$$
Observe that $I_q \subset B_{2s}(t_0)$ for all $q \geq 0$.

In the following, if not specified differently, every space norm is taken with respect to the sup in time localized in the interval $B_{2s}(t_0)$, i.e. for example: if $v \in L^\infty_t L^p_x$, we denote $\|v\|_{L^p} $ the quantity $\sup_{t\in B_{2s}(t_0)} \|v(t, \cdot ) \|_{L^p_x}$. We use $\lesssim$ as an inequality that holds up to a constant independent on $q$.\\

For $q \geq 0$, we want to guarantee 

\begin{subequations} \label{stima_norme_vq_Rq}
\begin{align}
\|v_q\|_{L^2} \leq 2 \|v_0 \|_{L^2} - \frac{\epsilon }{  4 \pi} \delta_{q}^{1/2} ,  \label{stima_norma_v_q}
 \\
\|\mathring{R}_q\|_{L^1} \leq  \lambda_q^{-3 \zeta} \delta_{q+1} , \label{stima_l1_su_Rq}
\\
\|v_q\|_{C^1_{x,t}(\T^3 \times B_{2s}(t_0))} \leq \lambda_q^4,  \label{stima_norma_v_q_C^1}
\end{align}
\end{subequations}
and moreover\footnote{Here Supp$_T(u)$ denotes the closure of $ \{ t \in (0,1): \exists x \in \T^3 \ \ u(x,t) \neq 0 \}.$}
\begin{subequations}  \label{stima_energia_supportoR_supporto_vq}
\begin{align}
\frac{\delta_{q+1}}{ \delta_1 \lambda_q^{\zeta/2}} \leq e(t) - \int_{\T^3} |v_q(x,t)|^2 dx \leq \frac{\delta_{q+1}  \epsilon_1}{\delta_1},  \text{ for all } t \in I_0,  \label{stima_energia}
\\
\text{Supp}_T (\mathring{R}_q) \subset I_q, \label{supporto_R_q}
\\
\text{Supp}_T (v_{q} - v_{q-1}) \subset I_q, \text{ for all } q \geq 1, \label{supporto_diff_v_q}
\end{align}
\end{subequations}
which are new with respect to the convex integration scheme proposed by Buckmaster and Vicol in \cite[Section 7]{BV19}.

\begin{prop} [Iterative Proposition] \label{p_iterative}
Let $e: [0,T ] \rightarrow (0, \infty)$ be a  strictly positive smooth function. For every $ \epsilon,s>0$ and $t_0\in(0,T)$ there exist $b>1$, $\beta(b)>0$, $\zeta >0$, $a_0= a_0( \beta, b, \zeta, e, \epsilon, s)$
such that 
for any $a \geq a_0$ which is a multiple of the geometric constant $n_\ast$ of Lemma \ref{l_geometric}, the following holds.
Let $(v_q,p_q, \mathring{R}_q)$ be a smooth triple solving the Navier-Stokes-Reynolds system \eqref{reynolds}  in $ \T^3\times B_{2s}(t_0)$ satisfying the inductive estimates \eqref{stima_norme_vq_Rq}-\eqref{stima_energia_supportoR_supporto_vq}.

Then there exists a second smooth triple  $(v_{q+1}, p_{q+1}, \mathring{R}_{q+1})$ which solves the Navier-Stokes-Reynolds system  in $ \T^3\times B_{2s}(t_0)$
\eqref{reynolds}, satisfies the estimates \eqref{stima_norme_vq_Rq}  and \eqref{stima_energia_supportoR_supporto_vq} at level $q+1$. In addition, we have that
\begin{align}  \label{stima_differenza_l2_vq}
\|v_{q+1}- v_q\|_{L^\infty(B_{2s}(t_0);L^2(\T^3))} \leq  \frac{\epsilon }{ \delta_1^{1/2} 4 \pi}\delta_{q+1}^{1/2}.
\end{align}
\end{prop}

\subsection{Proof of Theorem~\ref{t_main}}

{\it Step 1.  Let $v \in L^\infty ((0,T); L^2(\T^3))$ be a distributional solution of \eqref{NS}, such that $v \in C^{\infty}(\T^3 \times I)$, for some open interval $I \subset (0,T)$.
Then, we prove the following claim: for every $\epsilon >0$, there exists a distributional solution $v_{\epsilon} \in L^\infty ((0,T); L^2(\T^3))$ of \eqref{NS} such that
\begin{align} \label{p_main_inequality}
\| v_{\epsilon} - v \|_{L^\infty((0,T);(L^2(\T^3))} < \epsilon
\end{align}
and the kinetic energy of $v_{\epsilon}$
is strictly increasing in a sub-interval of $(0,T)$.\\

}

Let $t_0\in I$ and choose $s>0$ such that $\overline B_{2s}(t_0)\subset I$. Let $g \in C^\infty([0,T];[\frac{\epsilon_1}{2},\epsilon_1]) $ be such that

$$ g'(t_0) > \sup_{t \in (0,1)}  \left | \frac{d}{d t}  \int_{\T^3} |v(x,t)|^2 dx \right |,$$
and consider the kinetic energy (increasing in a neighbourhood of $t_0$)
\begin{align} \label{d_energia}
e(t):= \int_{\T^3} |v(x,t)|^2 dx + g(t).
\end{align}

Since the function $v$ is smooth in $\T^3\times I$ we consider the smooth solution $p$, with zero average, in $\T^3\times I$ of 
\eqref{eqn:p}, and define the starting triple $ (v_0, p_0, R_0) :=(v, p, {0})$.

Clearly $(v, p, {0})$ satisfies the estimates \eqref{stima_norme_vq_Rq} and \eqref{stima_energia_supportoR_supporto_vq} at step $q=0$, up to enlarge $a_0$\footnote{To be precise we considered $v_{-1}=v_0$.}, thus we can apply Proposition \ref{p_iterative} starting from the triple $ (v_0, p_0, R_0)$.
Hence, we get a sequence $\{ v_q \}_{q \in \N}$ that satisfies \eqref{stima_norme_vq_Rq}, \eqref{stima_energia_supportoR_supporto_vq} and moreover, from \eqref{stima_differenza_l2_vq} we get
 \begin{align}
\label{eqn:serieq}
 \sum_{q \geq 0} \|v_{q+1} - v_{q} \|_{L^2} &  \leq
\frac{\epsilon }{ \delta_1^{1/2} 4 \pi} \sum_{q \geq 0} \delta_{q+1}^{1/2} 
\leq 
\frac{\epsilon }{ \delta_1^{1/2} 4 \pi} \sum_{q \geq 0} (a^{- \beta b})^{q+1}\leq \frac{\epsilon}{2 (1- a^{- \beta b})}< \eps
 \end{align}
 where the last holds if $a_0$ is sufficiently large in order to have $a^{- \beta b} < 1/2$. Hence, there exists the limit $\tilde v_{\epsilon} := \lim_{q \rightarrow \infty} v_q$, in $L^\infty(B_{2s}(t_0); L^2(\T^3))$ such that $\|\tilde v_{\epsilon} - v \|_{L^\infty(B_{2s}(t_0); L^2(\T^3))}  <\eps$ and it is a distributional solution of the Navier-Stokes equations in $ B_{2s}(t_0)\times \T^3$, because by \eqref{stima_l1_su_Rq} we have that
 $\lim_{q \rightarrow \infty} \mathring{R}_q =0$ in $L^\infty(B_{2s}(t_0); L^1(\T^3))$.
One can verify that the vector field
\begin{align*}
 v_\epsilon=
\left\{\begin{array}{l}
\tilde v_\epsilon\quad \textit{in } B_{2s}(t_0)\\ 
v\quad \textit{in } [0,T]\setminus B_{2s}(t_0),
\end{array}\right.
\end{align*}
still solves \eqref{NS} in $[0,T]\times\T^3$ and satisfies \eqref{p_main_inequality}. Moreover the kinetic energy of $v_{\epsilon}$ is increasing in a neighbourhood of $t_0$ thanks to \eqref{stima_energia} and \eqref{d_energia}. \\

{\it Step 2. We conclude the proof of Theorem \ref{t_main}.}\\

Let $v_0$ be a distributional solution which is smooth in a subinterval of times and $\eps>0$; for instance, any Leray solution can be taken as $v_0$ since they are smooth outside a closed set of $\mathcal H^{1/2}$ measure $0$. 
We apply the Step 1 and get a distributional solution of Navier-Stokes $v_{\epsilon} \in L^\infty ((0,T); L^2(\T^3))$ such that
$\| v_\epsilon - v_1 \|_{L^\infty((0,T);L^2(\T^3))} < \epsilon$ with increasing kinetic energy in a sub-interval of $[0,T]$ and therefore such that $v_\epsilon \in \mathcal D \setminus \mathcal{L}$. 

Since $\mathcal L$ is closed with respect to $L^\infty L^2$ convergence, we deduce that the interior of $\overline{ \mathcal L}$ which coincides with the interior of  $\mathcal L$, is empty.

To show that $\mathcal{S}$ is a meagre set in $\mathcal{D}$, we rewrite it as 
$$ \mathcal{S} \subset \bigcup_{s \in \Q^+}   \bigcup_{t \in (0,1) \cap \Q } \{ v \in \mathcal{D} :  v \in C^\infty ((t -{s}, t+{s}) \times \T^3) \}
,$$
%
%
and we notice that from Step 1 the right-hand side is a countable union of nowhere dense sets, hence it is meagre. 

\section{Intermittent jets}\label{int_jets}

In this section we recall from \cite{BV19} the definition and the main properties of intermittent jets we will use in the convex integration scheme.
 
 \subsection{A geometric lemma.}
 We start with a geometric lemma. A proof of the following version, which is essentially due to De Lellis and Sz\'ekelyhidi Jr., can be found in  \cite[Lemma 4.1]{BCV}. This lemma allows us to reconstruct any symmetric $3 \times 3$ stress tensor $R$ in a neighbourhood of the identity as a linear combination of a particular basis.
 
 \begin{lemma} \label{l_geometric}
Denote by $\overline{B_{1/2}^{sym}} (Id)$ the closed ball of radius $1/2$ around the identity matrix in the space of symmetric $3 \times 3$ matrices.
There exists a finite set $\Lambda \subset \mathbb{S}^2 \cap \Q^3$ such that there exist $C^\infty $ functions $\gamma_{\xi} : B_{1/2}^{sym}(Id) \rightarrow \R $
which obey

$$R=  \sum_{\xi \in \Lambda} \gamma_\xi^2 ( {R}) \xi \otimes \xi, $$
for every symmetric matrix $R$ satisfying $|R-Id| \leq 1/2$. Moreover for each $\xi \in \Lambda$, let use define $A_{\xi} \in \mathbb{S}^2 \cap \Q^3$ to be an orthogonal vector to $\xi$. Then for each $\xi \in \Lambda$ we have that $\{ \xi, A_\xi , \xi \times A_\xi \} \subset \mathbb{S}^2 \cap \Q^3 $ form an orthonormal basis for $\R^3$.
Furthermore, since we will periodize functions,  let $n_\ast$ be the l.c.m. of the denominators of the rational numbers $\xi, A_\xi$ and $\xi \times A_\xi$, such that 
$$ \{ n_\ast \xi, n_\ast A_\xi , n_\ast \xi \times A_\xi \} \subset \Z^3.$$
\end{lemma}

\subsection{Vector fields}
 
Let $\Phi : \R^2 \rightarrow \R $ be a smooth function with support contained in a ball of radius 1. We normalize $\Phi$
such that $\phi = - \Delta \Phi$ obeys 

\begin{align} 
 \frac{1}{4 \pi^2} \int_{\R^2} \phi^2(x_1,x_2) dx_1 dx_2 =1 .
\end{align}

We remark that by definition $\phi$ has zero average. Define $\psi : \R \rightarrow \R $ to be a smooth,  zero average function with support in the ball of radius 1 satisfying 
\begin{align*}
\fint_{\T} \psi^2(x_3) dx_3 = \frac{1}{2 \pi} \int_{\R} \psi^2(x_3) dx_3 =1.
\end{align*}
We define the parameters $r_\perp$, $r_{||}$ and $\mu$ as follows
\begin{subequations} \label{d_parametri_r_perp_r_par}
\begin{align}
r_\perp := r_{\perp, q+1} := \lambda_{q+1}^{-6/7} (2 \pi )^{-1/7},
\\
r_{||} := r_{||, q+1} := \lambda_{q+1}^{-4/7},
\\
 \mu := \mu_{q+1}:= \lambda_{q+1}^{9/7} (2 \pi )^{1/7}.
\end{align}
\end{subequations}
We define $\phi_{r_\perp}, \Phi_{r_{\perp}}, $and $\psi_{r_{||}}$ 
to be the rescaled cut-off functions
\begin{align*}
\overline{\phi_{r_\perp}}(x_1,x_2):=\frac{1}{r_{\perp}} \phi \left (\frac{x_1}{r_{\perp}},\frac{x_2}{r_{\perp}} \right ),
\\
\overline{\Phi_{r_\perp}}(x_1,x_2):=\frac{1}{r_{\perp}} \Phi \left (\frac{x_1}{r_{\perp}},\frac{x_2}{r_{\perp}} \right ),
\\
\overline{\psi_{r_{||}}}(x_3):=\left (\frac{1}{r_{||}} \right )^{1/2} \psi \left (\frac{x_3}{r_{||}} \right ).
\end{align*}

With this rescaling we have $ \overline{\phi_{r_{\perp}}}= - r_{\perp}^2 \Delta \overline{\Phi_{r_{\perp}}}$. Moreover the functions $\overline{\phi_{r_{\perp}}}$ and $\overline{\Phi_{r_{\perp}}}$ are supported in the ball of radius $r_{\perp}$ in $\R^2$, $\psi_{r_{||}}$ is supported in the ball of radius $r_{||}$ in $\R$ and we keep the normalizations $\|\overline{\phi_{r_{\perp}}}\|_{L^2}^{2}= 4 \pi^2$ and $\|  \overline{\psi_{r_{||}}} \|_{L^2}^2=2 \pi.$

We then periodize the previous functions
\begin{align*}
{\phi_{r_\perp}}(x_1+ 2 \pi n,x_2 +  2 \pi m)= \overline{\phi_{r_\perp}}(x_1,x_2), 
\\
{\Phi_{r_\perp}}(x_1+ 2 \pi n,x_2 +  2 \pi m)= \overline{\Phi_{r_\perp}}(x_1,x_2), 
\\
{\psi_{r_{||}}}(x_3 + 2 \pi n)= \overline{\psi_{r_{||}}}(x_3).
\end{align*}
For every $\xi \in \Lambda$ (recalling the notations in Lemma \ref{l_geometric}), we introduce the functions defined on $\T^3 \times \R$
\begin{subequations} \label{d_psi_phi_Phi}
\begin{align} 
\psi_{\xi} (x,t) 
 := \psi_{r_{||}} (n_\ast r_{\perp } \lambda_{q+1} (x \cdot \xi + \mu t)), \\
\Phi_{\xi} (x) 
:= \Phi_{r_{\perp}} (n_\ast r_{\perp } \lambda_{q+1} (x- \alpha_\xi) \cdot A_\xi , n_\ast r_{\perp } \lambda_{q+1} (x- \alpha_\xi) \cdot (\xi \times A_\xi)) ,
\\
\phi_{\xi} (x) := \phi_{r_{\perp}} (n_\ast r_{\perp } \lambda_{q+1} (x- \alpha_\xi) \cdot A_\xi , n_\ast r_{\perp } \lambda_{q+1} (x- \alpha_\xi) \cdot (\xi \times A_\xi)) ,
\end{align}
\end{subequations}
where $\alpha_\xi$ are shifts which ensure that the functions $\{  \Phi_\xi \}$ have mutually disjoint support.

In order for such shifts $\alpha_\xi$ to exist, it is sufficient to assume that $r_{\perp}$ is smaller than a universal constant, which depends only on the geometry of the finite set $\Lambda$.

It is important to note that the function
$\psi_\xi$ oscillates at  frequency proportional to $r_{\perp}
r_{||}^{-1} \lambda_{q+1}$, whereas $\phi_{\xi}$ and $\Phi_\xi$  oscillate at frequency proportional to $\lambda_{q+1}$.

\begin{definition} \label{d_intermittent}
The intermittent jets are vector fields $W_\xi : \T^3 \times \R \rightarrow \R^3$ defined as
\begin{align*}
W_\xi (x,t) := \xi \psi_\xi (x,t) \phi_\xi (x).
\end{align*}
\end{definition}

If $\sigma:= r_\perp n_\ast \in \N$, thanks to the choice of $n_\ast$ in Lemma \ref{l_geometric} we have that $W_\xi$ has zero average in $\T^3$
and is $ \left (\frac{\T}{\sigma} \right )^3$ periodic. Moreover, by our choice of $\alpha_\xi$, we have that
\begin{align*}
W_\xi \otimes W_{\xi'} \equiv 0,
\end{align*}
whenever $\xi \neq \xi' \in \Lambda$,
i.e. $\{ W_\xi\}_{\xi \in \Lambda} $ have mutually disjoint support. The essential identities obeyed by the intermittent jets are
\begin{align}
\|W_\xi \|_{L^p(\T^3)}^p = \frac{1}{8 \pi^3} \| \psi_\xi \|_{L^p(\T^3)}^p \|\phi_\xi \|_{L^p(\T^3)}^p  \nonumber
\\
\text{ div} (W_\xi \otimes W_\xi) = 2 (W_\xi \cdot \nabla \psi_\xi) \phi_\xi \xi = \frac{1}{\mu} \partial_t (\phi_{\xi}^2 \psi_\xi^2 \xi) \label{identity_intermittent_deriv_temporale}
\\
\fint_{\T^3} W_{\xi} \otimes W_\xi =\xi \otimes \xi, \nonumber
\end{align}
where the last identity will be useful to apply Lemma \ref{l_geometric}.\\

We denote by $\PP$ the operator which projects a function onto its non-zero frequencies $\PP f = f - \fint_{\T^3} f$, and by $\PH$ we will denote the usual Helmholtz projector onto divergence-free vector fields, $\PH f = f - \nabla (\Delta^{-1} \diver f)$.
Motivated by \eqref{identity_intermittent_deriv_temporale}, we define
\begin{align} \label{d_intermittent_temporale_W_t}
W^{(t)}_\xi (x,t) := - \frac{1}{\mu} \PH \PP \phi_{\xi}^2 (x) \psi_\xi^2 (x,t) \xi.
\end{align}

Lastly, we note that the intermittent jets $W_\xi$ are not divergence free, then we introduce the following two functions $W_\xi^{(c)}, V_\xi : \T^3 \times \R \rightarrow \R^3$
\begin{subequations} 
\begin{align*} 
 & V_\xi (x,t) := \frac{1}{n_\ast \lambda_{q+1}^2 } \xi \psi_\xi(x,t) \Phi_{\xi}(x), \\
& W_\xi^{(c)} (x,t):=\frac{1}{n_\ast \lambda_{q+1}^2 }
\nabla \psi_\xi(x,t) \times (\nabla \times \Phi_\xi (x) \xi).
\end{align*}
\end{subequations}
Using $ \Delta \Phi_\xi =- \lambda_{q+1}^2 n_\ast^2 \phi_\xi$ we compute the intermittent jets in terms of $V_\xi$
\begin{align} 
\lambda_{q+1}^2 n_\ast^2  W_\xi & = \lambda_{q+1}^2 n_\ast^2  \xi \phi_\xi \psi_\xi= -   \Delta \Phi_\xi \psi_\xi \xi \notag \\
& =   \nabla \times (\psi_\xi \nabla \times (\Phi_\xi \xi)) - \nabla \psi_\xi  \times( \nabla \times \Phi_\xi \xi )   \notag \\
&=   \nabla \times \nabla \times (\psi_\xi \Phi_\xi \xi) - \nabla \times (\nabla \psi_\xi \times \Phi_\xi \xi) -
\nabla \psi_\xi \times (\nabla \times \Phi_\xi \xi) \notag
\\
&=
 \nabla \times \nabla \times (\psi_\xi \Phi_\xi \xi) - 
\nabla \psi_\xi \times (\nabla \times \Phi_\xi \xi)  \notag
\\
&= \lambda_{q+1}^2 n_\ast^2 \left ( \nabla \times \nabla \times V_\xi 
- 
W_\xi^{(c)} \right ), \label{conto_incompressibility}
\end{align}

from which we deduce
\begin{align*} 
\text{div} (W_\xi + W_\xi^{(c)}) \equiv 0.
\end{align*} 

Moreover, since $r_\perp \ll r_{||}$, the correction
$W_{\xi}^c$ is comparatively small in $L^2$ with respect to $W_\xi$, more precisely we state the following lemma (see \cite[Section 7.4]{BV_fen}).

\begin{lemma} \label{l_intermittent_inequalities}
For any $N,M \geq 0$ and $p \in [1, \infty]$ the following inequalities hold
\begin{subequations}  \label{disug_intermittent}
\begin{align}
\|\nabla^N \partial_t^M \psi_\xi \|_{L^p} \lesssim 
r_{||}^{1/p - 1/2} \left ( \frac{r_\perp \lambda_{q+1} }{r_{||}} \right)^N \left(  \frac{ r_{\perp} \lambda_{q+1} \mu}{r_{||}}   \right )^M
\\
\| \nabla^N \phi_{\xi}\|_{L^p} + \|\nabla^N \Phi_{\xi}\|_{L^p}
\lesssim
 r_{\perp}^{2/p -1 } \lambda_{q+1}^{N}
 \\
 \|\nabla^N \partial_t^M W_\xi \|_{L^p}  
 \lesssim
 r_\perp^{2/p-1} r_{||}^{1/p -1/2} \lambda_{q+1}^N 
 \left( \frac{r_{\perp} \lambda_{q+1} \mu}{r_{||}}\right)^M
 \\
  \frac{r_{||}}{r_\perp} \| \nabla^N \partial_t^M W_\xi^{(c)} \|_{L^p} \lesssim r_\perp^{2/p-1} r_{||}^{1/p -1/2} \lambda_{q+1}^N 
 \left( \frac{r_{\perp} \lambda_{q+1} \mu}{r_{||}}\right)^M
 \\
 \lambda_{q+1}^2 \|\nabla^N \partial_t^M V_\xi \|_{L^p}
 \lesssim
 r_\perp^{2/p-1} r_{||}^{1/p -1/2} \lambda_{q+1}^N 
 \left( \frac{r_{\perp} \lambda_{q+1} \mu}{r_{||}}\right)^M.
\end{align}
\end{subequations}

The implicit constants are independent of $\lambda_{q+1}, r_\perp, r_{||}, \mu$. 
\end{lemma}

\section{Proof of the iterative proposition}

 
 Given $(v_q,p_q, \mathring{R}_q)$ a triple solving the Navier-Stokes-Reynolds system \eqref{reynolds} in $ \T^3 \times B_{2s}(t_0) $ satisfying the inductive estimates \eqref{stima_norme_vq_Rq} and \eqref{stima_energia_supportoR_supporto_vq}  at step $q$, we have to construct $(v_{q+1}, p_{q+1}, \mathring{R}_{q+1})$ which still solves the Navier-Stokes-Reynolds system  
\eqref{reynolds}  in $ \T^3 \times B_{2s}(t_0) $ and satisfies the estimates \eqref{stima_norme_vq_Rq} and \eqref{stima_energia_supportoR_supporto_vq}  at step $q+1$ and the estimate \eqref{stima_differenza_l2_vq} holds.

\subsection{Mollification}
In order to avoid a loss of derivatives in the iterative scheme, we replace $v_q$ by a mollified velocity field $\tilde{v}_\ell$. For this purpose we choose a small parameter $\ell \in (0,1)$ which lies between $\lambda_q^{-1}$ and $\lambda_{q+1}^{-1}$ and that satisfies
\begin{subequations} \label{parametro_ell_condizioni}
\begin{align*}
\ell \lambda_q^4 \leq \lambda_{q+1}^{- \alpha} \\
\ell^{-1} \leq \lambda_{q+1}^{2 \alpha},
\end{align*}
\end{subequations}
where $0< \alpha \ll 1$. This can be done since $\alpha b >4$.

For instance, we may define $\ell$ as the geometric mean of the two bounds imposed before
$$\ell = \lambda_{q+1}^{-3 \alpha/2} \lambda_q^{-2}.$$

With this choice we also have that $\ell \ll s_{q+1}.$
Let $ \{ \theta_\ell \}_{\ell>0}$ and $\{ \varphi_\ell \}_{\ell>0}$  be two standard families of Friedrichs mollifiers on $\R^3$ (space) and  $\R$ (time) respectively.
We define the mollification of $v_q$ and $\mathring{R}_{q}$ in space and time, at length scale $\ell$ by
\begin{align*}
\overline{v}_\ell := (v_q \ast_x \theta_\ell) \ast_t \varphi_\ell,
\\
\mathring {\overline{R}}_{\ell} := (\mathring{R}_q \ast_x \theta_\ell) \ast_t \varphi_\ell,
\end{align*}
where we possibly extend to $0$ the definition of $v_q$ outside $B_{2s}(t_0)$.
We have that $\overline{v}_\ell$ solves
\begin{align} \label{reynolds_bar_v_ell}
\left\{\begin{array}{l}
\partial_t \overline{v}_\ell + \text{div} (\overline{v}_\ell \otimes \overline{v}_\ell) + \nabla p_\ell - \nu \Delta \overline{v}_\ell = \text{div}( \mathring{\overline{R}}_\ell + \mathring{\overline{R}}_{com})  \\ 
\diver  \overline{v}_\ell=0,
\end{array}\right.
\end{align}

where $\mathring{\overline{R}}_{com}$ is defined by
\begin{align*}
\mathring{\overline{R}}_{com}= (\overline{v}_\ell \mathring{\otimes} \overline{v}_\ell) - ((v_q  \mathring{\otimes} v_q ) \ast_x \theta_\ell) \ast_t \varphi_\ell .
\end{align*}

We introduce the following notations $y + I_q := (t_0 - S_q - y , t_0 + S_q + y)$ and $\tilde{I}_q := \frac{s_{q+1}}{2} +  I_q$. Let $\eta \in C^{\infty}_c (\tilde I_q; \R^+)$ such that
\begin{align*}
\eta (t) \equiv 1 \text{ for all } t \in I_q,
\\
\| \eta \|_{C^N} \leq C \left(\frac{2}{s} \right)^{Nq},
\end{align*}
Moreover, we define
\begin{align*}
\tilde{v}_\ell = \eta \overline{v}_\ell  + (1- \eta )  v_q .
\end{align*}
Note that $\tilde{v}_\ell$ satisfies
\begin{align*} 
\text{Supp}_T(\tilde{v}_\ell - v_q) \subset \tilde{I}_q \subset I_{q+1},
\end{align*}
that will be crucial in order to guarantee \eqref{supporto_diff_v_q} at step $q+1$.

Moreover, using \eqref{reynolds_bar_v_ell} and that $(v_q, p_q, \mathring{R}_q)$ is a Navier--Stokes--Reynolds solution, we have that $\tilde{v}_\ell$ satisfies

\begin{align*}
\partial_t \tilde{v}_\ell + \text{div}(\tilde{v}_\ell \otimes \tilde{v}_\ell) - \Delta \tilde{v}_\ell & =
  ( \overline{v}_\ell - v_q) \partial_t \eta 
  +
 \eta (1- \eta) \text{div} (\overline{v}_\ell \mathring{\otimes} (v_q - \overline{v}_\ell ))
\\
& + \eta (1- \eta ) \text{div} (v_q \mathring{\otimes} (\overline{v}_\ell - v_q))
\\
& + \eta \text{div} ({\overline{R}}_\ell + {\overline{R}}_{com} ) + (1- \eta) \text{div}(\mathring{R}_q) - \nabla \pi_\ell,
\end{align*}
for some pressure $\pi_\ell$.

We recall the inverse divergence operator from \cite{DLSZ13}.
\begin{definition}\label{d_reynoldsoperator}
We define the Reynolds operator $\RR :  C^{\infty}(\T^3; \R^3) \rightarrow C^{\infty}(\T^3; \R^3)$ as
 $$\RR v := \frac{1}{4}( \nabla \PH \Delta^{-1} v + (\nabla \PH \Delta^{-1} v)^T) + \frac{3}{4} ( \nabla \Delta^{-1} v + (\nabla \Delta^{-1} v)^{T}) -\frac{1}{2} \text{div }(\Delta^{-1}v \text{Id}),$$
 for every smooth $v$ with zero average. If $v \in C^{\infty}(\T^3; \R^3)$ we define $\RR v := \RR (v- \fint_{\T^3} v).$
\end{definition}

We have the following
\begin{prop}[$\RR= \text{div}^{-1}$] \label{p_reynoldsdiv}
For any $v \in C^{\infty}(\T^3; \R^3)$ with zero average we have
\begin{enumerate}
\item $\RR v (x)$ is a symmetric traceless matrix, for each $x \in \T^3$,
\item $\diver \RR v= v - \fint_{\T^3} v$,
\item $\RR $ can be extended to a continuous operator from $L^p$ to $L^p$,
\item $\RR \nabla $ can be extended to a continuous operator from $L^p$ to $L^p$.
\end{enumerate}
\end{prop}

Using \eqref{supporto_R_q} and that $\eta (t) \equiv 1$ on $I_q$,  we have
$$ (1- \eta) \text{div} (\mathring{R}_q) \equiv 0.$$
Thus $\tilde{v}_\ell$ solves
\begin{align*}
\partial_t \tilde{v}_\ell + \text{div}(\tilde{v}_\ell \otimes \tilde{v}_\ell) - \Delta \tilde{v}_\ell + \nabla \pi_\ell =
\text{div} ({{R}}_\ell + {{R}}_{com} + {{R}}_{loc}),
\end{align*}
where ${{R}}_\ell  = \eta \mathring {\overline{R}}_\ell$, 
${{R}}_{com} = \eta \mathring {\overline{R}}_{com} $ and
\begin{align*}
{R}_{loc} :=  
\eta (1- \eta) \overline{v}_\ell \mathring{\otimes} (v_q - \overline{v}_\ell )
 +
  \eta (1- \eta )  v_q \mathring{\otimes} (\overline{v}_\ell - v_q)
  +
  \RR \left ( ( \overline{v}_\ell - v_q) \partial_t \eta  \right ).
\end{align*}

A simple bound on $\overline{v}_\ell - v_q$ on $L^\infty_t L^2$ is given by
$$\| \overline{v}_\ell - v_q \|_{ L^2 } \lesssim \ell \|v_q\|_{C^1} \leq \ell \lambda_q^4 \ll \frac{1}{10} \lambda_{q+1}^{-4 \zeta} \delta_{q+2},$$
where the last holds if $4 \zeta + 2 \beta b < \alpha.$
Then using the previous bound, \eqref{stima_norma_v_q} and that   $ \| \RR \|_{L^2 \rightarrow L^2} \lesssim 1$ by Proposition \ref{p_reynoldsdiv}, we have
\begin{align*}
\| {{R}}_{com} \|_{L^1} +  \| {{R}}_{loc} \|_{L^1} \ll \frac{1}{3} \lambda_{q+1}^{-3 \zeta} \delta_{q+2},
\end{align*}
where we used that $\lambda_{q+1}^{\zeta} \gg C \left(\frac{2}{s} \right)^{q}$, unless to possibly enlarge $a_0(s, \zeta)$.
Note that we also have the property on the compact supports of the errors
\begin{align*}
\text{Supp}({{R}}_\ell ) \cup   \text{Supp} ({{R}}_{com}) \cup \text{Supp} ({{R}}_{loc})  \subset \tilde{I}_q \subset I_{q+1}.
\end{align*}

The mollified functions satisfy
\begin{subequations} \label{stima_mollificazione}
\begin{align}
&\| \tilde{v}_\ell \|_{C^N_{x,t} (\T^3 \times B_{2s}(t_0))} 
\lesssim \lambda_q^4 \ell^{-N+1} \lesssim \lambda_q^{- \alpha } \ell^{-N}, \text{ } N \geq 1\label{stima_v_moll}, 
 \\
&\|\tilde{v}_\ell\|_{L^2}  \leq \|v_q\|_{L^2} + \| v_q - \overline{v}_\ell \|_{L^2} \leq 2 \|v_0\|_{L^2} - \delta_q^{1/2} + \lambda_q^{-\alpha},
\\
& \| \tilde{v}_\ell - v_q \|_{L^2} \lesssim \ell \lambda_q^4 \leq \lambda_{q+1}^{- \alpha}, 
\\
 & \|R_\ell \|_{L^1} \leq  \lambda_q^{-3 \zeta} \delta_{q+1}, 
\\
& \| R_\ell \|_{C^N_{x,t}} \lesssim \lambda_q^{-3 \zeta} \delta_{q+1} \ell^{-4 -N}, \text{ } N \geq 0 \label{stima_R_moll}.
\end{align}
\end{subequations}
We are now ready to go to the perturbation step, in which we will add a small perturbation to $\tilde{v}_\ell$ in order to cancel the bigger error ${{R}}_\ell $ proving  \eqref{stima_l1_su_Rq}, \eqref{supporto_R_q} and satisfying all the other estimates \eqref{stima_norme_vq_Rq}, \eqref{stima_energia_supportoR_supporto_vq} and \eqref{stima_differenza_l2_vq}. 

\subsection{Amplitudes}
Here we define the amplitudes of the perturbation, namely the functions needed to apply Lemma \ref{l_geometric} and cancel the Reynolds error $R_\ell$.
We define $\chi: \R^+ \rightarrow \R^+$, a smooth function such that
$$\chi (z):=
\begin{cases}
1 & if  \hspace{0.4cm } 0 \leq z \leq 1 \\
z & if \hspace{0.4cm} z \geq 2
\end{cases}$$
and $z \leq 2 \chi(z) \leq 4 z$ for $z \in (1,2)$ and $\chi(z) \geq 1$ for all $z \in [0, \infty)$.
We define for all $t \in I_0=[t_0-\frac s2,t_0+\frac s2]$

\begin{align} \label{d_rho_bar}
\overline{\rho} (t) :=  \frac{1}{ 3 \int_{\T^3} \chi  \left (\frac{|R_\ell (x,t )| 4 \lambda_q^{\zeta} \delta_1}{\delta_{q+1}} \right ) dx} \left ( e(t) - \int_{\T^3} |\tilde{v}_\ell(x,t)|^2 dx - \frac{\delta_{q+2}}{2} \right )
\end{align}
and with a little abuse of notation we define
\begin{align*}
\overline{\rho} (t) := \overline{\rho}  \left (t_0 + \frac{s}{2}  \right ) \text{ for all }  t> t_0 + \frac{s}{2},
\\
\overline{\rho} (t) := \overline{\rho} \left (t_0 - \frac{s}{2} \right ) \text{ for all }  t< t_0 - \frac{s}{2}.
\end{align*}

 Now, we consider another local cut-off in time $ \tilde{\eta} \in C^{\infty}_c(I_{q+1}; \R^+)$ such that
\begin{align*}
\tilde{\eta} (t) \equiv 1 \text{ for all } t \in \tilde{I}_q,
\\
\| \tilde{\eta} \|_{C^N} \leq C \left(\frac{2}{s} \right)^{Nq},
\end{align*} 
and we define
 \begin{align} \label{d_rho}
\rho (x,t) :=  \tilde{\eta}^2(t)  \overline{\rho} (t) \chi \left( \frac{|R_\ell(x,t)| 4 \lambda_q^{\zeta} \delta_1 }{\delta_{q+1}}\right ).
 \end{align}

\begin{lemma}
The following estimates hold

\begin{align}
\frac{\delta_{q+1}}{\delta_1 \lambda_q^{\zeta} }  \leq \overline{\rho} (t) \leq \frac{\epsilon_1 \delta_{q+1}}{\delta_1}\label{disug_rho_t}, 
\\
\left | \frac{R_\ell (x,t)}{ \rho (x,t)} \right |  \leq \frac{1}{2}\label{quotient_R_rho}, 
\\
\| \rho \|_{L^1} \leq   16 \pi^3\epsilon_1\frac{\delta_{q+1}}{\delta_1}. \label{disug_rho_L^p}
\end{align}
\end{lemma}

\begin{proof}
Note that 
\begin{equation}\label{diff_norme_l2}
\left | \| v_q \|_{L^2}^2 - \| \tilde{v}_{\ell} \|_{L^2}^2  \right | \leq
 \| v_q - \tilde{v}_\ell \|_{L^2} \|v_q+ \tilde{v}_\ell \|_{L^2} \lesssim 
 \ell \| v_q \|_{C^1} \| v_q \|_{L^2}   \lesssim \ell \lambda_q^4 \leq \lambda_q^{- \zeta} \delta_{q+1},
\end{equation}
where in the last inequality we used that $2 \beta + \frac{\zeta}{b} < \alpha$. Moreover, thanks to the construction of $\chi$ and \eqref{stima_l1_su_Rq} we have
\begin{equation}\label{stima_chi_R}
(2 \pi )^3 \leq  \int_{\T^3} \chi  \left (\frac{|R_\ell (x,t )|4 \lambda_q^{\zeta} \delta_1}{\delta_{q+1}} \right ) dx  \leq 2 (2 \pi )^3.
\end{equation}
Thus, thanks to \eqref{stima_energia}, \eqref{diff_norme_l2} and \eqref{stima_chi_R} we get
\begin{align*}
\overline{\rho} (t) & \leq \frac{1}{3 \cdot (2 \pi)^3} 
 \left ( e(t) - \int_{\T^3} |v_q(x,t)|^2 dx  \right) + \frac{1}{3 \cdot (2 \pi)^3}  \left(  \int_{\T^3} |v_q(x,t)|^2 dx  -\int_{\T^3} |\tilde{v}_\ell(x,t)|^2 dx - \frac{\delta_{q+2}}{2} \right )
  \\
  & \leq 
  \frac{1}{3 \cdot (2 \pi)^3} \left  (2 \frac{\delta_{q+1}  \epsilon_1}{\delta_1} \right ) \leq \frac{\epsilon_1 \delta_{q+1}}{\delta_1}.
\end{align*}
and similarly
\begin{align*}
\overline{\rho} (t)  & \geq \frac{1}{6 \cdot (2 \pi)^3}  \left ( e(t) - \int_{\T^3} |v_q(x,t)|^2 dx  \right)+ \frac{1}{6 \cdot (2 \pi)^3}  \left(  \int_{\T^3} |v_q(x,t)|^2 dx  -\int_{\T^3} |\tilde{v}_\ell(x,t)|^2 dx - \frac{\delta_{q+2}}{2} \right )
\\
& \geq \frac{1}{6 \cdot (2 \pi)^3} \left  (\frac{\delta_{q+1}}{ \delta_1 \lambda_q^{\zeta/2}}-  \frac{\delta_{q+1}}{ \lambda_q^{ \zeta}} - \frac{\delta_{q+2}}{2} \right ) \geq \frac{\delta_{q+1}}{ \delta_1 \lambda_q^{ \zeta}},
\end{align*}
 where the last holds if we choose $a_0 (\zeta)$ sufficiently large. Thus \eqref{disug_rho_t} holds.
 
The proof of \eqref{quotient_R_rho} follows from the following computation, observing that $\text{Supp}_T (R_\ell) \subset \tilde{I}_q$, $\tilde{\eta} (t) \equiv 1$ for all $t \in \tilde{I}_q $ and that $\chi (z) \geq z/2$ for all $z \geq 0$
 
 \begin{align*}
 \left | \frac{R_\ell (x,t)}{ \rho (x,t)} \right | & \leq
 \frac{|R_\ell(x,t)|}{\overline{\rho}(t) \frac{|R_\ell(x,t)|}{2 \delta_{q+1}} 4 \lambda_q^{ \zeta} \delta_1 } =\frac{\delta_{q+1}}{2 \overline{\rho}(t) \lambda_q^{\zeta} \delta_1} \leq 1/2.
 \end{align*}
 
We conclude the proof by estimating
 \begin{align*}
 \int_{\T^3} |\rho (x,t)| dx 
& \leq \int_{  \frac{|R_\ell (x,t)| 4 \lambda_q^{ \zeta} \delta_1}{\delta_{q+1}} <1} |\rho (x,t)| dx + \int_{  \frac{|R_\ell (x,t)| 4 \lambda_q^{ \zeta} \delta_1}{\delta_{q+1}} \geq 1} |\rho (x,t)| dx 
\\
& \leq 8\pi^3 \left ( \delta_{q+1} \frac{\epsilon_1}{\delta_1} \right ) + 
\int_{ \T^3}   |8 \lambda_{q}^\zeta \epsilon_1 R_\ell| dx \\
&\leq  8 \pi^3 \left ( \delta_{q+1} \frac{\epsilon_1}{\delta_1} \right ) + 8 \epsilon_1 \lambda_q^{2 \zeta} \| R_\ell \|_{L^1}\\
&\leq 8\pi^3\epsilon_1 \left(\frac{1}{\delta_1}+\lambda_q^{-\zeta} \right)\delta_{q+1}\leq 16\pi^3\epsilon_1 \frac{\delta_{q+1}}{\delta_1}.
 \end{align*}
\end{proof}
We can now define the amplitudes functions $a_\xi : \T^3 \times (0,T) \rightarrow \R$ as 
\begin{align} \label{d_amplitudes}
a_\xi (x,t) := a_{\xi, q+1} (x,t) := \rho^{1/2}(x,t) \gamma_\xi \left (Id - \frac{R_\ell(x,t)}{\rho(x,t)} \right ),
\end{align}
where $\gamma_\xi$ are defined in Lemma \ref{l_geometric},
hence we also get the identity
\begin{align} \label{identity_geometric}
\rho (x,t ) Id - {R_\ell(x,t)} = \sum_{\xi \in \Lambda} a_\xi^2 (x,t)  \xi \otimes \xi.
\end{align}
\begin{lemma} \label{l_amplitudes}
The following estimates hold
\begin{align}
\|a_\xi \|_{L^2} \leq \frac{\delta_{q+1}^{1/2}}{2 C_0 |\Lambda|} \frac{\epsilon}{4 \pi \delta_{1}^{1/2}} ,
\\
\|a_\xi \|_{C^N_{x,t}} \lesssim  \ell^{-8-5N},  
\end{align}
where $C_0$ is the universal constant for which Lemma \ref{l_decorellation} holds.
\end{lemma}
\begin{proof}
We define 
\begin{align*}
\rho_1 (x,t)&:= \overline{\rho} (t) \chi \left( \frac{|R_\ell(x,t)| 4 \lambda_q^{\zeta} \delta_1 }{\delta_{q+1}}\right ),
\\
\overline{a}_\xi (x,t) &:= \rho_1^{1/2}(x,t) \gamma_\xi \left (Id - \frac{R_\ell(x,t)}{\rho(x,t)} \right ),\\
a_\xi (x,t) &= \tilde{\eta}(t) \overline{a}_\xi (x,t).
\end{align*}

The first estimate follows from \eqref{disug_rho_L^p} and the definition of $\epsilon_1$ 
\begin{align*}
\|a_\xi \|_{L^2}  & \leq \| \rho \|_{L^1}^{1/2} \| \gamma_\xi \|_{C^0} \| \tilde{\eta} \|_{C^0} \leq  \left( 16 \pi^3 \delta_{q+1} \frac{\epsilon_1}{\delta_1} \right )^{1/2}  \| \gamma_\xi \|_{C^0}\leq \frac{\delta_{q+1}^{1/2}}{2 C_0 |\Lambda|} \frac{\epsilon}{4 \pi \delta_{1}^{1/2}}.
\end{align*}
We prove the second estimate. 
We introduce the notation $\tilde{\gamma}_\xi (x,t) :=  \gamma_\xi \left (\Id - \frac{R_\ell(x,t)}{\rho(x,t)} \right )$ and
thanks to Proposition \ref{p_disug_prodotto_holder} we have
\begin{align*}
\| \overline{a}_\xi \|_{C^N_{x,t}} \lesssim \| \rho_1^{1/2} \|_{C^N} \| \tilde{\gamma} \|_{C^0} + 
\| \rho_1^{1/2} \|_{C^0} \| \tilde{\gamma}\|_{C^N}.
\end{align*}
We now estimate every piece.
Using Proposition \ref{p_disug_composizione_holder} and \eqref{stima_energia}
$$\| \overline{\rho} \|_{C^N_t} \lesssim  \ell^{- 5N}.$$
Thanks to the previous inequality, Proposition \ref{p_disug_composizione_holder} and  Proposition \ref{p_disug_prodotto_holder} we get
\begin{align} \label{rho_stima_C_N}
\| \rho_1 \|_{C^N_{x,t}} \lesssim  \ell^{-4-5 N}.
\end{align}
Using Proposition \ref{p_disug_composizione_holder}, estimate \eqref{stima_R_moll},  the previous estimate and that $\rho$ is bounded from below by $\frac{\delta_{q+1}}{\delta_1 \lambda_q^{\zeta} }$, we have
\begin{align*}
\| \tilde{\gamma}\|_{C^N} \lesssim  \left  \| \frac{R_\ell }{\rho}  \right \|_{C^N} 
\lesssim \ell^{-8- 5 N}
\end{align*}
and using also that $\frac{\delta_{q+1}}{\delta_1 \lambda_q^{\zeta} } \geq \ell$ (choosing $\zeta = \zeta(\alpha)$ sufficiently small), we have
\begin{align*}
\| \rho^{1/2}_1 \|_{C^N_{x,t}} \lesssim \ell^{-5-5 N}.
\end{align*}
Hence
\begin{align*}
\| \overline{a}_\xi  \|_{C^N_{x,t}} \lesssim \ell^{-8-5N} .   
\end{align*}
Moreover, by
applying Proposition \ref{p_disug_prodotto_holder} we get
\begin{align*}
\| a_\xi \|_{C^N_{x,t}} \lesssim \| \overline{a}_\xi  \|_{C^N_{x,t}}  \|  \tilde{\eta} \|_{C^0} + \|  \tilde{\eta} \|_{C^N} \|\overline{a}_\xi  \|_{C^0_{x,t}}
\lesssim \| \overline{a}_\xi  \|_{C^N_{x,t}},
\end{align*}
since $s_{q+1}^{-1} \ll \lambda_q \ll \ell^{-1}$, up to enlarge $a_0(s, \alpha)$.
\end{proof}

\subsection{Principal part of the perturbation, incompressibility and temporal correctors}
The principal part of $w_{q+1}$ is defined as
\begin{align} \label{d_w_q+1_p}
w_{q+1}^{(p)} := \sum_{\xi \in \Lambda} a_{\xi} W_\xi .
\end{align}
The incompressibility corrector $w_{q+1}^{(c)}$, that we define in order to have the incompressibility of $w_{q+1}$, is defined as
\begin{align*} 
w_{q+1}^{(c)}:=  \sum_{\xi \in \Lambda} \curl( \nabla a_\xi \times V_\xi ) +\nabla a_\xi \times \curl V_{\xi} + a_{\xi} W_{\xi}^{(c)}.
\end{align*}
Note that
\begin{align*}
w_{q+1}^{(p)} + w_{q+1}^{(c)} = \sum_{\xi \in \Lambda} \nabla \times \nabla \times ( a_\xi  V_\xi ), \\
\text{div} (w_{q+1}^{(p)} + w_{q+1}^{(c)})=0,
\end{align*}
where the first equation follows from a direct computation similar to  \eqref{conto_incompressibility} with amplitudes functions
\begin{align*}
 a_\xi W_\xi & = a_\xi  \nabla \times \nabla \times V_\xi 
- 
a_\xi W_\xi^{(c)} 
\\
& = \nabla \times ( a_\xi \nabla \times V_\xi ) - \nabla a_\xi \times (\nabla \times V_\xi ) - 
a_\xi W_\xi^{(c)}
\\
& = \nabla \times \nabla \times ( a_\xi  V_\xi ) - \nabla \times (\nabla a_\xi \times V_\xi)- \nabla a_\xi \times (\nabla \times V_\xi ) - 
a_\xi W_\xi^{(c)}.
\end{align*}
Moreover, we introduce a temporal corrector similar to \eqref{d_intermittent_temporale_W_t} with amplitude functions
\begin{align} \label{w_q_temporale}
w_{q+1}^{(t)} := - \frac{1}{\mu} \sum_{\xi \in \Lambda} \PH  \PP \left ( a_\xi^2 \phi_\xi^2 \psi_\xi^2 \xi \right).
\end{align}
Note that $w_{q+1}^{(t)}$ satisfies
\begin{align*}
\partial_t w_{q+1}^{(t)} &+ \sum_{\xi \in \Lambda} \PP \left (a_{\xi}^2 \text{div}  (W_\xi \otimes W_\xi) \right ) 
\\
& =
- \frac{1}{\mu} \sum_{\xi \in \Lambda} \PH \PP  \partial_t \left ( a_\xi^2  \phi_\xi^2 \psi_\xi^2 \xi \right) 
+
 \frac{1}{\mu} \sum_{\xi \in \Lambda}  \PP \left ( a_\xi^2 \partial_t \left( \phi_\xi^2 \psi_\xi^2 \xi \right) \right) \\
 & =
 \underbrace{
 (\text{Id} - \PH) \frac{1}{\mu} \sum_{\xi \in \Lambda}  \PP  \partial_t \left ( a_\xi^2  \phi_\xi^2 \psi_\xi^2 \xi \right) }_{=:\nabla P_{q+1}}
 -
 \frac{1}{\mu} \sum_{\xi \in \Lambda}  \PP \left (\partial_t a_\xi^2  \left( \phi_\xi^2 \psi_\xi^2 \xi \right) \right) .
\end{align*}
From this computation and the identity \eqref{identity_geometric}, it follows that
\begin{align} \label{conto_per_R_oscillation}
\text{div} & ( w_{q+1}^{(p)} \otimes w_{q+1}^{(p)} + {R_\ell}) + \partial_t w_{q+1}^{(t)}  = \sum_{\xi \in \Lambda} \text{div} \left ( a_\xi^2 \PP \left( W_\xi \otimes W_\xi \right) \right ) +\nabla \rho + \partial_t w_{q+1}^{(t)} \notag
\\
&=
\sum_{\xi \in \Lambda} \PP \left (\nabla a_\xi^2 \PP \left( W_\xi \otimes W_\xi \right) \right ) +\nabla \rho 
+\sum_{\xi \in \Lambda} \PP \left ( a_\xi^2 \text{div} \left( W_\xi \otimes W_\xi \right) \right )
+ \partial_t w_{q+1}^{(t)} \notag
\\
&=
\sum_{\xi \in \Lambda} \PP \left (\nabla a_\xi^2 \PP \left( W_\xi \otimes W_\xi \right) \right ) +\nabla \rho + \nabla P_{q+1}
- \frac{1}{\mu} \sum_{\xi \in \Lambda} \PP \left ( \partial_t a_\xi^2  \left( \phi_\xi^2 \psi_\xi^2 \xi \right) \right ).
\end{align}

\subsection{The velocity increment and proof of the inductive estimates}

We now define the total increment 
\begin{align} 
w_{q+1}:= w_{q+1}^{(p)} + w_{q+1}^{(c)} + w_{q+1}^{(t)}
\end{align}
and the new vector field is then given by
 \begin{align} 
 v_{q+1}:= \tilde{v}_\ell + w_{q+1}.
 \end{align}
 In this section we verify that the inductive estimates \eqref{stima_norme_vq_Rq} hold with $q$ replaced by $q+1$, and that \eqref{stima_differenza_l2_vq} is satisfied.

 \subsubsection{Proof of \eqref{stima_differenza_l2_vq}}
We want to apply Lemma  \ref{l_decorellation} in $L^2$ with $f= a_\xi$ and $g_\sigma = W_\xi$, which is by
construction $\left (\frac{\T}{\sigma} \right )^3-$periodic  with $\sigma \sim  \lambda_{q+1} r_\perp$, where $\sim$ means up to a constant depending only on $n_\ast$ and $\xi \in \Lambda$. For this purpose, note that by \eqref{l_amplitudes} we get
  $$\|D^j a_\xi \|_{L^2} \leq  \frac{\delta_{q+1}^{1/2}}{2 C_0 |\Lambda|} \frac{\epsilon}{4 \pi \delta_{1}^{1/2}}
  \ell^{-13 j},$$
 and thus we can take $C_f= \frac{\delta_{q+1}^{1/2}}{2 C_0 |\Lambda|} \frac{\epsilon}{4 \pi \delta_{1}^{1/2}}$.
 By conditions on $\ell$ we have 
$\ell^{-13} \leq \lambda_{q+1}^{26 \alpha}  $, whereas by \eqref{d_parametri_r_perp_r_par} we have that $\lambda_{q+1} r_\perp = \left ( \frac{\lambda_{q+1} }{2 \pi} \right )^{1/7}$. Thus, since $\alpha < \frac{1}{7 \cdot 70}$ 
and $a$ is huge, Lemma \ref{l_decorellation} is applicable. Combining the resulting estimate with the normalization  $\|W_\xi \|_{L^2} =1$ we obtain
\begin{align} \label{stima_w_q+1_L2}
\|w_{q+1}^{(p)} \|_{L^2} \leq \sum_{\xi \in \Lambda} \frac{ C_0 \delta_{q+1}^{1/2}}{2 C_0 |\Lambda|} \frac{\epsilon}{4 \pi \delta_{1}^{1/2}} \|W_\xi \|_{L^2} \leq  \frac{\epsilon}{4 \pi \delta_{1}^{1/2}} \frac{1}{2} \delta_{q+1}^{1/2}.
\end{align}
For the correctors $w_{q+1}^{(c)}$ and $w_{q+1}^{(t)}$ we can use rougher  estimates since they are considerably smaller than $w_{q+1}^{(p)}$.
 The following estimates are consequence of  Proposition \ref{p_reynoldsdiv}, estimates  \eqref{d_parametri_r_perp_r_par},  \eqref{disug_intermittent} and Lemma \ref{l_amplitudes}
 \begin{subequations}\label{stima_w_q_p_w_q_c_w_q_t}
 \begin{align} 
\|w_{q+1}^{(p)}\|_{L^p}  & \lesssim \sum_{\xi \in \Lambda} \| a_\xi\|_{C^0} \|W_\xi\|_{L^p}  \lesssim  \ell^{-8} r_{\perp}^{2/p -1 } r_{||}^{1/p-1/2}
\\
\notag \\
\|w_{q+1}^{(c)}\|_{L^p}  & \lesssim \sum_{\xi \in \Lambda} \|a_\xi\|_{C^2} \|V_\xi \|_{W^{1,p}} + \|a_\xi \|_{C^0} \|W_\xi^{(c)} \|_{L^p}  \notag
\\
&\lesssim  \ell^{-18 } r_\perp^{2/p -1 } r_{||}^{1/p-1/2} \lambda_{q+1}^{-1} +  \ell^{-8} r_\perp^{2/p -1 } r_{||}^{1/p-1/2} \frac{r_\perp}{r_{||}} \notag
\\
& \lesssim  \ell^{-18}  r_\perp^{2/p -1 } r_{||}^{1/p-1/2} \lambda_{q+1}^{-2/7}
\\
\notag \\
\|w_{q+1}^{(t)}\|_{L^p} & \lesssim \mu^{-1} \sum_{\xi \in \Lambda}                \|a_{\xi}\|_{C^0}^2 \|\phi_\xi\|_{L^2p}^2 \|\psi_\xi \|_{L^2p}^2  \lesssim  \ell^{-16} r_\perp^{2/p-2} r_{||}^{1/p-1} \mu^{-1}  \notag \\
& \lesssim 
\ell^{-16} r_\perp^{2/p-1} r_{||}^{1/p-1/2} \lambda_{q+1}^{-1/7},
 \end{align}
 \end{subequations}
 where in the last inequality we used also the continuity of $\PH$ in $L^p$ (for any $1<p< \infty$) and the fact that $\| \phi_\xi^2 \psi_\xi^2 \|_{L^p}= \| \phi_\xi^2 \|_{L^p} \| \psi_\xi^2 \|_{L^p}$, thanks to Fubini.
 
 Combining \eqref{stima_w_q+1_L2}, with the last two estimates of \eqref{stima_w_q_p_w_q_c_w_q_t} for $p=2$, and using \eqref{d_parametri_r_perp_r_par}, we obtain for a constant $C>0$ (which is independent of $q$) that\footnote{In the last inequality, we have implicitly used that $\alpha< 1/(7 \cdot 74)$ and $a_0$ be sufficiently large.}
\begin{align*}
\|w_{q+1} \|_{L^2} & \leq  \left (\frac{\epsilon}{4 \pi \delta_{1}^{1/2}} \frac{1}{2} \delta_{q+1}^{1/2} + C \ell^{-18} \frac{r_\perp}{r_{||}}  + C \ell^{-16} \lambda_{q+1}^{-1/7} \right ) 
\\
& \leq \frac{\epsilon}{4 \pi \delta_{1}^{1/2}} \left ( \frac{\delta_{q+1}^{1/2}}{2} + C \lambda_{q+1}^{36 \alpha -2/7} + C \lambda_{q+1}^{32 \alpha -1/7} \right )  \leq  \frac{3}{4} \frac{\epsilon}{4 \pi \delta_{1}^{1/2}} \delta_{q+1}^{1/2}.
\end{align*}
Moreover from \eqref{stima_mollificazione}, by choosing $a_0$ sufficiently large we get
 $$\| v_{q+1} - v_q \|_{L^2} \leq \|w_{q+1} \|_{L^2} + \| \tilde{v}_\ell - v_q \|_{L^2} \leq \frac{\epsilon}{4 \pi \delta_{1}^{1/2}} \delta_{q+1}^{1/2},$$
  thus \eqref{stima_differenza_l2_vq} is satisfied.\\
 
 \subsubsection{Proof of \eqref{stima_norma_v_q}} The bound \eqref{stima_norma_v_q} follows easily from  and the previous estimates (if $q \neq 0$)
\begin{align*}
\|v_{q+1}\|_{L^2} &= \|v_{q+1} - v_q + v_q\|_{L^2}\leq \|v_q\|_{L^2} +  \|v_{q+1}- v_q\|_{L^2} 
\\
& \leq 2 \| v_0 \|_{L^2} - \frac{\epsilon }{4 \pi} \delta_{q}^{1/2} + \frac{\epsilon }{ \delta_1^{1/2} 4 \pi} \delta_{q+1}^{1/2} \leq 2 \| v_0 \|_{L^2} - \frac{\epsilon }{  4 \pi} \delta_{q+1}^{1/2},
\end{align*}
 where in the last inequality we have used that $a$ is taken sufficiently large and $b\gg 1$. If $q=0$, then \eqref{stima_norma_v_q} is trivial.\\
 
 \subsubsection{Proof of \eqref{supporto_diff_v_q}}  The property \eqref{supporto_diff_v_q} is verified since
 \begin{align*}
 v_{q+1} - v_q = \tilde{v}_\ell - v_q + w_{q+1}
 \end{align*}
 and Supp$_T(\tilde{v}_\ell - v_q) \subset \text{Supp}_T \eta \subset I_{q+1}$ ,  $\text{Supp}_T w_{q+1} \subset \text{Supp}_T a_\xi \subset \text{Supp}_T \tilde{\eta} \subset I_{q+1}$.\\
 
 \subsubsection{Proof of \eqref{stima_norma_v_q_C^1}} Taking either a spatial or a temporal derivative, using Lemma \ref{l_intermittent_inequalities}, Lemma \ref{l_amplitudes}, \eqref{d_parametri_r_perp_r_par} and \eqref{parametro_ell_condizioni}, we have
 \begin{align*}
 \|w_{q+1}^{(p)} \|_{C^1_{x,t}} & \lesssim \|a_{\xi}\|_{C^1_{x,t}} \|W_\xi \|_{C^0_{x,t}} + \|a_{\xi}\|_{C^0_{x,t}} \|W_\xi \|_{C^1_{x,t}} 
 \\
 & \lesssim 
 \ell^{-13} r_\perp^{-1} r_{||}^{-1/2} + \ell^{-8} r_\perp^{-1} r_{||}^{-1/2} \lambda_{q+1}^{2} \lesssim \lambda_{q+1}^{2+ 8/7 + 26 \alpha},
 \\
 \\
 \|w_{q+1}^{(p)} \|_{C^1_{x,t}} & \lesssim \|a_\xi \|_{C^2_{x,t}} \|V_\xi \|_{C^1_{x,t}} + \|a_\xi \|_{C^1_{x,t}} \|W_\xi^{(c)} \|_{C^1_{x,t}},
 \\
 &\lesssim \ell^{18} r_\perp^{-1} r_{||}^{-1/2} \lambda_{q+1}^{-2} \lambda_{q+1}^2  + \ell^{-13} \frac{r_{\perp}}{r_{||}} r_{\perp}^{-1} r_{||}^{-1/2} \lambda_{q+1}^{2}  \lesssim \lambda_{q+1}^{2+6/7 + 36\alpha},
 \\
 \\
 \|w_{q+1}^{(t)} \|_{C^1_{x,t}} & \lesssim \|w_{q+1}^{(t)} \|_{C^{1,\alpha}_{x,t}}  \lesssim \frac{1}{\mu} \|a_\xi^2 \phi_\xi^2 \psi_\xi^2 \|_{C^{1, \alpha}_{x,t}} 
 \\
 & \lesssim \frac{1}{\mu} \|a_\xi^2\|_{C^0_{x,t}} \|\phi_\xi^2\|_{C^0_{x,t}} \| \psi_\xi^2 \|_{C^{1, \alpha}_{t}} \lesssim \frac{1}{\mu} \ell^{-16} r_{\perp}^{-2} r_{||}^{-1/2} \lambda_{q+1}^{2} \lambda_{q+1}^{2\alpha} \lesssim \lambda_{q+1}^{3 -2/7 + 34 \alpha}.
 \end{align*}
In the latter inequality we have used that $\PH$ is continuous on H{\"o}lder spaces. Therefore, using that $\alpha <1/40$, that $a_0$ is sufficiently large and thanks to estimate \eqref{stima_v_moll}, we have
 $$ \|v_{q+1}\|_{C^1_{x,t}(B_{2s}(t_0) \times \T^3)} \leq \|\tilde{v}_\ell\|_{C^1_{x,t}(B_{2s}(t_0) \times \T^3)} + \|w_{q+1}\|_{C^1_{x,t}} \leq \lambda_{q+1}^{4}. $$
\subsection{The new Reynolds stress} 
\vspace{0.5cm}
Here we will define the new Reynolds stress $\mathring{R}_{q+1}$. By definitions, $\tilde{v}_{q+1}$ solves
\begin{align*}
\text{div } & \mathring{R}_{q+1} - \nabla p_{q+1}  \notag
\\
&= 
\partial_t  (\tilde{v}_\ell + w_{q+1}) + \diver ((\tilde{v}_\ell + w_{q+1}) \otimes (\tilde{v}_\ell + w_{q+1}))  -  \Delta (\tilde{v}_\ell + w_{q+1}) 
\notag \\
& = \underbrace{
-  \Delta w_{q+1} + \partial_t (w_{q+1}^{(p)} + w_{q+1}^{(c)}) + \text{div} (\tilde{v}_\ell \otimes w_{q+1} + w_{q+1} \otimes \tilde{v}_\ell)
}_{\text{div} (R_{lin}) + \nabla p_{lin}} \notag
\notag \\
& + \underbrace{ \text{div} \left ( (w_{q+1}^{(c)} + w_{q+1}^{(t)}) \otimes w_{q+1}  +  w_{q+1}^{(p)}  \otimes (w_{q+1}^{(c)} + w_{q+1}^{(t)})  \right )
}_{\text{div} (R_{cor}) + \nabla p_{cor}}
\notag \\
& + \underbrace{
\text{div} (w_{q+1}^{(p)} \otimes w_{q+1}^{(p)} + {R_\ell}) + \partial_t w_{q+1}^{(t)}
}_{\text{div}(R_{osc}) + \nabla p_{osc}}
+ \text{div}(R_{com}) - \nabla p_\ell.
\end{align*}
More precisely
\begin{align*}
R_{lin} & := -  \RR \Delta w_{q+1} + \RR \partial_t (w_{q+1}^{(p)} +w_{q+1}^{(c)}) + \tilde{v}_\ell \mathring{\otimes} w_{q+1} +
w_{q+1} \mathring{\otimes} \tilde{v}_\ell,
\\
R_{cor}& := \left ( w_{q+1}^{(c)} + w_{q+1}^{(t)} \right) \mathring{\otimes} w_{q+1}  + w_{q+1}^{(p)}
 \mathring{\otimes} \left ( w_{q+1}^{(c)} + w_{q+1}^{(t)} \right),
 \\
 R_{osc} & := \sum_{\xi \in \Lambda} \RR \left ( \nabla a_{\xi}^2 \PP (W_\xi \otimes W_\xi ) \right) - \frac{1}{\mu} \sum_{\xi \in \Lambda} \RR \left ( \partial_t a^2_\xi (\phi_\xi^2 \psi_\xi^2 \xi ) \right),
 \\
 p_{lin} &:= 2 \tilde{v}_\ell \cdot w_{q+1},
 \\
 p_{cor} &:= |w_{q+1}|^2 - |w_{q+1}^{(p)}|^2,
 \\
 p_{osc} &:= \rho + P_{q+1},
\end{align*}
where  the definitions of $p_{osc}$ and $R_{osc}$ are justified by the previous computation \eqref{conto_per_R_oscillation}.
Hence we define
$$p_{q+1}:= p_\ell - p_{cor} - p_{lin} - p_{osc}$$
and
$$ \mathring{R}_{q+1}:= R_{lin} + R_{cor} + R_{osc} + R_{com} + R_{loc},$$
where the last two were defined during the mollification step.
We observe that the new Reynolds-stress $\mathring{R}_{q+1}$ is traceless, this property will be crucial in the energy estimates.
\subsection{Estimates for the new Reynolds stress}
We  need to estimate the new stress $\mathring{R}_{q+1}$ in $L^1$. However, since the Calder{\'o}n-Zygmund operator $\nabla \RR$  fails to be bounded on $L^1$, we introduce an integrability parameter,
$$p \in (1,2] \text{ such that } p-1 \ll 1.$$
 Recalling the parameters choice \eqref{d_parametri_r_perp_r_par}, we fix $p$ to obey
\begin{align} \label{parametro_p_scelta_piccolo}
r_{\perp}^{2/p -2} r_{||}^{1/p -1} \leq (2\pi)^{1/7} \lambda_{q+1}^{16(p-1)/(7p)} \leq \lambda_{q+1}^{\alpha},
\end{align}
where we recall that $0<\alpha < \frac{1}{7 \cdot 74}$.
For instance, we take $p= \frac{32}{32- 7 \alpha}$.

\subsubsection{Linear error Reynolds stress} 
By using Proposition \ref{p_reynoldsdiv} we get that
\begin{align*}
\|R_{lin} \|_{L^p}
& \lesssim  \| \RR \Delta w_{q+1}\|_{L^p} +
\|\tilde{v}_\ell \mathring{\otimes} w_{q+1} +
w_{q+1} \mathring{\otimes} \tilde{v}_\ell \|_{L^p}
+  \| \RR \partial_t (w_{q+1}^{(p)} +w_{q+1}^{(c)}) \|_{L^p}
\\
& \lesssim
 \| \nabla w_{q+1}\|_{L^p} +
 \|\tilde{v}_\ell \|_{L^\infty} \|w_{q+1}\|_{L^p} +
 \sum_{\xi \in \Lambda} \| \partial_t \curl (a_\xi V_\xi) \|_{L^p}
 \\
 & \lesssim
 \sum_{\xi \in \Lambda} \|a_{\xi}\|_{C^1} \|W_\xi \|_{W^{1,p}}
 +
 \|\tilde{v}_\ell\|_{C^1} \sum_{\xi \in \Lambda} \|a_{\xi}\|_{C^1} \|W_\xi \|_{W^{1,p}} \\
 & +
 \sum_{\xi \in \Lambda} (\|a_\xi \|_{C^1} \| \partial_t V_\xi \|_{W^{1,p}} + \|\partial_t a_\xi \|_{C^1} \|V_\xi \|_{W^{1,p}}).
\end{align*}
Thus, by appealing to Lemma \ref{l_intermittent_inequalities}, Lemma \ref{l_amplitudes}, estimates \eqref{stima_mollificazione}   and to the choice of $p= \frac{32}{32- 7 \alpha}$, we conclude
\begin{align*}
\|R_{lin} \|_{L^p} &\lesssim 
\ell^{-13} r_\perp^{\frac{2}{p}-1 } r_{||}^{\frac{1}{p}-1/2} \lambda_{q+1}+\ell^{-18 } r_\perp^{\frac{2}{p}-1 } r_{||}^{\frac{1}{p}-\frac{1}{2}} \lambda_{q+1}
+ \ell^{-18 } \lambda_{q+1}^{-1} r_\perp^{\frac{2}{p}-1 } r_{||}^{\frac{1}{p}-\frac{1}{2}} 
\\
& \lesssim 
\ell^{-18} \lambda_{q+1}^{\alpha} \lambda_{q+1} r_\perp r_{||}^{1/2}  \lesssim \lambda_{q+1}^{37 \alpha  -\frac{1}{7}} \ll \frac{1}{6} \lambda_{q+1}^{-3 \zeta} \delta_{q+2},
\end{align*}
where for the last inequality we used that $\alpha <\frac{1}{7 \cdot 74}$ and $2 \beta b + 3 \zeta < \frac{1}{14}$.
\subsubsection{Corrector error}
The estimate on the corrector error is a consequence of (\ref{stima_w_q_p_w_q_c_w_q_t}) and our choice of $p$\begin{align*}
\| R_{cor} \|_{L^p} & \leq 
\| w_{q+1}^{(c)} + w_{q+1}^{(t)} \|_{L^{2p}} \|w_{q+1}\|_{L^{2p}}
+
\|w_{q+1}^{(p)}\|_{L^{2p}}  \|w_{q+1}^{(c)} + w_{q+1}^{(t)}\|_{L^{2p}}
\\
& \leq
2 \| w_{q+1}^{(c)} + w_{q+1}^{(t)} \|_{L^{2p}} \|w_{q+1}\|_{L^{2p}}
\\
& \lesssim \ell^{-18} r_\perp^{1/p -1} r_{||}^{\frac{1}{2p} - \frac{1}{2}} \lambda_{q+1}^{-1/7} \lesssim \lambda_{q+1}^{36 \alpha +  \frac{\alpha}{2} - 1/7} \ll \frac{1}{6} \lambda_{q+1}^{-3 \zeta} \delta_{q+2},
\end{align*}
where the last inequality is justified as before.
\subsubsection{Oscillation error}
By using  the boundedness on $L^p$ of the Reynolds operator $\RR$, Lemma \ref{l_intermittent_inequalities}, Lemma \ref{l_amplitudes},  \eqref{d_parametri_r_perp_r_par}, Fubini (to separate $\phi_\xi$ and $\psi_\xi$) and the choice of $p$ we can estimate the second summand in the definition of $R_{osc}$ as
\begin{align*}
 \left  \|\frac{1}{\mu} \sum_{\xi \in \Lambda} \RR \left ( \partial_t a^2_\xi (\phi_\xi^2 \psi_\xi^2 \xi ) \right) \right \|_{L^p}
& \leq 
\mu^{-1} \sum_{\xi \in \Lambda}                \|a_{\xi}\|_{C^1}^2 \|\phi_\xi\|_{L^2p}^2 \|\psi_\xi \|_{L^2p}^2  \lesssim  \mu^{-1} \ell^{-26} \lambda_{q+1}^{\alpha} \ll \frac{1}{6} \lambda_{q+1}^{-3 \zeta} \delta_{q+2}.
\end{align*}
To estimate the remaining summand we will use Lemma \ref{p_reynolds_prodotto_1/k}. We apply it with
$a= \nabla a_\xi^2$, $\kappa= \sigma= \lambda_{q+1} r_\perp$ and $\mathbb{P}_{\geq \sigma} (f)= \PP (W_\xi \otimes W_\xi )$, that is a $\frac{\T^3}{\sigma}-$periodic function.
Then we have
\begin{align*}
\left \| \sum_{\xi \in \Lambda} \RR \left ( \nabla a_{\xi}^2 \PP (W_\xi \otimes W_\xi ) \right) \right \|_{L^p}
& \lesssim ( \lambda_{q+1} r_\perp)^{-1} \| \PP (W_\xi \otimes W_\xi ) \|_{L^p}
\|\nabla a_\xi^2 \|_{C^1}
\\
& \lesssim \ell^{-21}  \lambda_{q+1}^{-1/7} \|W_\xi \|_{L^{2p}}^2 
\lesssim   \ell^{-21}  \lambda_{q+1}^{-1/7}  r_\perp^{\frac{2}{p}-1 } r_{||}^{\frac{1}{p}-\frac{1}{2}}
\\
& \lesssim  \lambda_{q+1}^{42 \alpha + \alpha -1/7} \ll \frac{1}{6} \lambda_{q+1}^{-3 \zeta} \delta_{q+2}.
\end{align*}\\

Then \eqref{stima_l1_su_Rq} at step $q+1$ follows easily using also the previous estimates for $R_{com}$ and $R_{loc}$ 
\begin{align*}
\| \mathring{R}_{q+1} \|_{L^1} & \leq
\|R_{lin}\|_{L^1} + \|R_{cor}\|_{L^1} + \|R_{osc}\|_{L^1} + \|R_{com}\|_{L^1} + \|R_{loc} \|_{L^1} 
\\
& \leq
\frac{2 }{3} \lambda_{q+1}^{-3 \zeta} \delta_{q+2} + \frac{1 }{3} \lambda_{q+1}^{-3 \zeta} \delta_{q+2}  \leq  \lambda_{q+1}^{-3 \zeta} \delta_{q+2},
\end{align*}
where in the last inequality we have used that $2 \beta b + 3 \zeta < \alpha$.
Finally, since $\text{Supp}_T w_{q+1} \subset I_{q+1}$, then also  \eqref{supporto_R_q} holds at step $q+1$.

\subsection{Energy estimate}
In order to complete the proof of Proposition \ref{p_iterative} we only need to prove the energy estimate \eqref{stima_energia} at step $q+1$.
\begin{lemma} 
The following estimate holds for all $t \in I_0$
\begin{align}\label{stima_energia_finale}
\frac{\delta_{q+2}}{\lambda_{q+1}^{\zeta/2}} \leq e(t) - \int_{\T^3} |v_{q+1}(x,t)|^2 dx \leq \frac{\delta_{q+2}  \epsilon_1}{\delta_1}.
\end{align}
\end{lemma}
\begin{proof}
Recalling \eqref{d_w_q+1_p} and the mutually disjoint supports of $\{W_\xi\}_{\xi \in \Lambda}$ we notice that
\begin{align} \label{l_energia_conto1}
|w_{q+1}^{(p)}|^2 & = \left  | \sum_{\xi \in \Lambda} a_\xi W_\xi  \right |^2 = \sum_{\xi \in \Lambda} \text{Tr} (a_\xi W_\xi \otimes  a_\xi W_\xi ) \notag
\\
& = \sum_{\xi \in \Lambda} a_\xi^2 \text{Tr} \left ( \fint_{\T^3} W_\xi \otimes W_\xi \right ) 
+
\sum_{\xi \in \Lambda} a_\xi^2 \text{Tr} \left (W_\xi \otimes W_\xi - \fint_{\T^3} W_\xi \otimes W_\xi \right ) \notag
\\
& = 3 \rho + \sum_{\xi \in \Lambda} a_\xi^2 \text{Tr} \left (W_\xi \otimes W_\xi - \fint_{\T^3} W_\xi \otimes W_\xi \right ),
\end{align}
where in the last equation we used the traceless property of $R_\ell$ and \eqref{identity_geometric}.

Applying Lemma \ref{l_media_piccola}  with $f$ replaced by $a_\xi^2$ (which oscillates at frequency $\sim \ell^{-5}$),  the constant $C_f \sim \ell^{- 16}$ (thanks to the estimate of Lemma \ref{l_amplitudes}) and $g_\sigma$ replaced with $W_\xi \otimes W_\xi - \fint_{\T^3} W_\xi \otimes W_\xi $ (where $\sigma = \lambda_{q+1} r_\perp$), we get
\begin{align}\label{l_energia_conto2}
\left |  \int_{\T^3} \sum_{\xi \in \Lambda} a_\xi^2 \text{Tr} \left (W_\xi \otimes W_\xi - \fint_{\T^3} W_\xi \otimes W_\xi \right ) \right | \lesssim \ell^{-21 } \frac{1}{\lambda_{q+1} r_\perp} \ll \frac{\delta_{q+2}}{6},
\end{align}
where in the last inequality we used that $\alpha < \frac{1}{7 \cdot 74}$ and $  2 \beta b  < \frac{1}{14}$.
We write the identity
\begin{align} 
e(t) - \int_{\T^3} |v_{q+1}|^2 & = e(t) - \left ( \int_{\T^3} |\tilde{v}_{\ell}|^2 + \int_{\T^3} |w_{q+1}^{(p)}|^2 \right )  - \left ( \int_{\T^3} |w_{q+1}^{(c)} + w_{q+1}^{(t)} |^2 + 2 \int_{\T^3} \tilde{v}_\ell \cdot w_{q+1}  \right )  \notag
\\
&- \left( 2 \int_{\T^3} w_{q+1}^{(p)} \cdot  (w_{q+1}^{(c)} + w_{q+1}^{(t)}) \right) \label{l_energia_conto3}
\end{align}
and thanks to \eqref{l_energia_conto1}, \eqref{l_energia_conto2} and to the definition of $\rho$ \eqref{d_rho}, using also that $\tilde{\eta} \equiv 1$ in $I_0$, we have
\begin{align*}
\frac{\delta_{q+2}}{\lambda_{q+1}^{\zeta/4}} \leq e(t) - \left ( \int_{\T^3} |\tilde{v}_{\ell}|^2 + \int_{\T^3} |w_{q+1}^{(p)}|^2 \right ) \leq  \frac{2 \delta_{q+2}}{3}, \text{ for all } t \in I_0,
\end{align*}
up to possibly enlarge $a_0(\zeta)$.
Moreover, by using \eqref{stima_mollificazione} and \eqref{stima_w_q_p_w_q_c_w_q_t} we can estimate
\begin{align*}
\left | \int_{\T^3} |w_{q+1}^{(c)} + w_{q+1}^{(t)} |^2 + 2 \int_{\T^3} \tilde{v}_\ell \cdot w_{q+1}  \right | \leq \frac{\delta_{q+2}}{\lambda_{q+1}^{\zeta/3}},
\\
\left | 2 \int_{\T^3} w_{q+1}^{(p)} \cdot  (w_{q+1}^{(c)} + w_{q+1}^{(t)}) \right | \leq \frac{\delta_{q+2}}{\lambda_{q+1}^{\zeta/3}},
\end{align*}
from which \eqref{stima_energia_finale} follows.
\end{proof}
\appendix
\section{Useful tools}

In this section we state some useful results needed in the convex integration scheme.
\begin{prop} \label{p_disug_composizione_holder}
Let $\Psi : \Omega \rightarrow \R $ and $u : \R^n \rightarrow \Omega$ be two smooth functions, with $ \Omega \subset \R^N$. Then, for every $m \in \N^+$, there exists a constant $C>0$ (depending only on $m, N, n$) such that
\begin{align*}
[\Psi \circ u ]_m \leq C( [\Psi]_1  [u]_m  + \|D \Psi \|_{C^{m-1}} \|u\|_{C^0}^{m-1} [u]_m ), 
\\
[\Psi \circ u ]_m \leq C( [\Psi]_1  [u]_m  + \|D \Psi \|_{C^{m-1}} \|u\|_{C^1}^{m}  ), 
\end{align*}
where $[f]_m = \max_{|\beta=m|} \|D^\beta f\|_{0}$.
\end{prop}

\begin{prop} \label{p_disug_prodotto_holder}
Let $f,g : \T^3 \rightarrow \R $ be two smooth real value functions. 
For any integer  $r \geq  0$ there exists a constant $C>0$, depending only on $r$ such that
\begin{align*}
[f g ]_r \leq C ( [f]_r  \|g\|_{C^0} + \|f\|_{C^0} [g]_r),
\end{align*}
where $[f]_m = \max_{|\beta=m|} \|D^\beta f\|_{0}$.
\end{prop}
The following lemma is essentially Lemma 3.7 in \cite{BV19}.
\begin{lemma} \label{l_decorellation}
Fix integers $N, \sigma \geq 1$ and let $ \zeta >1 $ such that
\begin{align}
\frac{2 \pi \sqrt{3} \zeta}{\sigma} \leq \frac{1}{3} \hspace{0.3cm} and \hspace{0.3cm} \zeta^4 
\frac{(2 \pi \sqrt{3} \zeta)^N}{\sigma^N} \leq 1.
\end{align} 
Let $p \in \{1,2\}$ and let $f, g \in C^{\infty}(\T^3; \R^3)$. Suppose that there exists a constant $C_{f}>0 $ such that
$$ \|\nabla^j f \|_{L^p} \leq C_f \zeta^j,$$
holds for all $0 \leq j \leq N+4$. Then we have that 
$$\|f g_\sigma\|_{L^p} \leq C_0 C_f \|g_{\sigma}\|_{L^p},$$
where $C_0$ is a universal constant.
\end{lemma}

The following lemma is essentially Lemma B.1 in \cite{BV19}.
\begin{lemma} \label{p_reynolds_prodotto_1/k}
Fix $\kappa \geq 1$, $p \in (1,2]$, and a sufficiently large $L \in \N$. Let $a \in C^{L} (\T^3)$ be such that there exists $1 \leq \lambda \leq \kappa$, $C_a >0$ with
\begin{align*}
\| D^j a \|_{L^\infty} \leq C_a \lambda^j,
\end{align*}
for all $0 \leq j \leq L$. Assume furthermore that $\int_{\T^3} a(x) \mathbb{P}_{\geq \kappa} f(x) dx =0$. Then we have
\begin{align*}
\| | \nabla |^{-1} (a \mathbb{P}_{\geq \kappa} (f)) \|_{L^p} \lesssim C_a \left ( 1 + \frac{\lambda^L}{\kappa^{L-2}} \right ) \frac{\|f\|_{L^p}}{\kappa}
\end{align*}
for any $f \in L^p(\T^3)$, where the implicit constant depends on $p$ and $L$.
\end{lemma}
\begin{lemma} \label{l_media_piccola}
Let $g : \T^3 \rightarrow \R$ such that
$$\fint_{\T^3} g(x)dx =0,$$
 and let $g_\sigma : \T^3 \rightarrow \R$: $g_\sigma (x):= g (\sigma x)$.
 Let $f : \T^3 \rightarrow \R$ such that
 $$\| \nabla f \|_{C^0} \leq C_f \zeta,$$
 then we have
 $$\left | \int_{\T^3} g_\sigma(x) f(x) dx  \right | \lesssim \frac{C_f \zeta}{ \sigma} \|g_\sigma \|_{L^1(\T^3)},$$
 where $\lesssim$ means up to a universal constant.
\end{lemma}

\end{document}